\documentclass[11pt,longtable,letterpaper,oneside,reqno]{amsart}
\usepackage{amsmath,amsfonts,amsthm,amssymb,amscd,stmaryrd,url,hyperref,subfig,cancel,todonotes}

\usepackage{amsrefs}
\newenvironment{rezabib}
  {\bibdiv\biblist\setupbib}
  {\endbiblist\endbibdiv}
  
  \def\setupbib{\catcode`@=\active}
\begingroup\lccode`~=`@
  \lowercase{\endgroup\def~}#1#{\gatherkey{#1}}
\def\gatherkey#1#2{\gatherkeyaux{#1}#2\gatherkeyaux}
\def\gatherkeyaux#1#2,#3\gatherkeyaux{\bib{#2}{#1}{#3}}  

\usepackage{longtable,tabularx,tabulary}
\usepackage{titletoc}
\usepackage{float,mathtools}
\usepackage{color}
\usepackage{anysize,placeins,enumitem}
\usepackage{latexsym}
\usepackage{graphicx}
\usepackage{graphicx}
\usepackage{epsfig}
\usepackage{booktabs}
\usepackage{xcolor}
\usepackage{textgreek}
\usepackage{textcomp}
\usepackage{cases}
\usepackage{url}

\usepackage{todonotes}
\usepackage{mathrsfs}
\usepackage{caption}

 \DeclareMathOperator{\N}{N}

\newcommand{\R}{{\mathbb{R}}}
\newcommand{\C}{{\mathbb{C}}}

\renewcommand{\Re}{{\mathfrak{Re}}}
\renewcommand{\Im}{{\mathfrak{Im}}}

 \DeclareMathOperator{\modd}{mod}

\renewcommand{\b}{\beta}
\newcommand{\g}{\gamma}

\newtheorem{theorem}{Theorem}[section]
 \newtheorem{corollary}[theorem]{Corollary}
 \newtheorem{lemma}[theorem]{Lemma}
 
 \newtheorem{proposition}[theorem]{Proposition}
 \newtheorem{defn}[theorem]{Definition}

  \theoremstyle{remark}

\numberwithin{equation}{section}
\reversemarginpar 

\begin{document}

\title[Counting zeros of Dedekind zeta functions]{Counting zeros of Dedekind zeta functions}

\author[E. Hasanalizade]{Elchin Hasanalizade}
\address{Department of Mathematics and Computer Science\\
University of Lethbridge\\
4401 University Drive\\
Lethbridge, Alberta\\
T1K 3M4 Canada}
\email{e.hasanalizade@uleth.ca}

\author[Q. Shen]{Quanli Shen}
\address{Department of Mathematics and Computer Science\\
University of Lethbridge\\
4401 University Drive\\
Lethbridge, Alberta\\
T1K 3M4 Canada}
\email{quanli.shen@uleth.ca}

\author[P.J. Wong]{Peng-Jie Wong}
\address{Department of Mathematics and Computer Science\\
University of Lethbridge\\
4401 University Drive\\
Lethbridge, Alberta\\
T1K 3M4 Canada}
\email{pengjie.wong@uleth.ca}

\thanks{This research was supported by the NSERC Discovery grants RGPIN-2020-06731 of Habiba Kadiri and  RGPIN-2020-06032 of Nathan Ng. P.J.W. was supported by a PIMS postdoctoral fellowship and the University of Lethbridge.}

\begin{abstract}
Given a number field $K$ of degree $n_K$ and with absolute discriminant $d_K$, we obtain an explicit  bound for the number $N_K(T)$ of non-trivial zeros (counted with multiplicity), with height at most $T$, of the Dedekind zeta function $\zeta_K(s)$ of $K$. More precisely, we show that for $T \geq 1$,
$$
\Big| N_K (T)  - \frac{T}{\pi} \log \Big( d_K \Big( \frac{T}{2\pi e}\Big)^{n_K}\Big)\Big|
 \le    0.228 (\log d_K + n_K \log T) + 23.108  n_K  + 4.520,
$$
which improves previous results of  Kadiri and Ng, and Trudgian.  The improvement is based on ideas  from the recent work of Bennett \emph{et al.} on counting zeros of Dirichlet $L$-functions.
\end{abstract}




\maketitle

\section{Introduction}
Given a number field $K$, the Dedekind zeta function $\zeta_K(s)$ of $K$ is defined by 
\[
\zeta_K (s) = \sum_{\mathfrak{a}\neq 0} \frac{1}{\N(\mathfrak{a})^s},
\]
for $\Re(s)>1$, where the sum is over non-zero integral ideals of $K$. It is known that $\zeta_K (s)$ has an analytic continuation to a meromorphic function on $\Bbb{C}$ with only a simple pole at $s = 1$, and its zeros $\rho= \b + i \g$ encode deep arithmetic information of $K$. For instance, the generalised Riemann hypothesis,  asserting that if $\zeta_K(\rho) = 0$ and $\beta\in(0,1)$, then $\beta = \frac{1}{2}$, leads to the strongest form of the prime ideal theorem. A related prominent question is to count the zeros of $\zeta_K(s)$ in the critical strip $0<\Re(s)<1$.  For $T\geq 0$,   we set 
\[
N_K(T) = \# \{ \rho \in  \C  \mid  \zeta_K(\rho) =0,\  0  <\beta <  1, \ |\gamma| \leq T\},
\]
counted with multiplicity if there are any multiple zeros.
The estimate of $N_K(T)$ is crucial for proving  effective versions of the Chebotarev density theorem as well as bounding the least prime in the Chebotarev density theorem (see \cite{LMO, LO}).  Moreover, to make these results explicit, it is natural to further require a determination of the implied constants for the estimate of $N_K(T)$.

Adapting the arguments of Backlund \cite{Ba18},  McCurley \cite{McC84}, and  Rosser \cite{Ro41}, in \cite{KN12}, Kadiri  and  Ng showed that for $T\ge 1$, one has
\begin{equation}\label{counting-zeros}
 \Big| N_K (T)  -  \frac{T}{\pi} \log \Big( d_K\Big( \frac{T}{2\pi e}\Big) ^{n_K} \Big)  \Big|  \leq  
D_1 (\log d_K + n_K \log T) +D_2  n_K  +D_3,
\end{equation}
with admissible $(D_1,D_2,D_3)= (0.506,16.950, 7.663 )$,
where $n_K$ and $d_K$ are the degree and absolute discriminant of $K$, respectively; also, $D_1$ can be taken as small as $(\pi \log 2)^{-1}\approx 0.459$ at expense of larger $D_2  n_K  + D_3$.  This was improved by  Trudgian \cite{Tr15} (not only for Dedekind zeta functions but also for Dirichlet $L$-functions). In particular, as asserted in \cite{Tr15}, the estimate \eqref{counting-zeros} is valid with $(D_1,D_2,D_3)= (0.316, 5.872, 3.655)$, and  the constant $D_1$ in \eqref{counting-zeros} could be made as small as 0.247 (with larger $D_2  n_K  + D_3$).  Unfortunately, as pointed out by  Bennett, Martin, O'Bryant, and Rechnitzer  \cite{BMOR20}, there is an error in \cite{Tr15} that appears as  the ranges of  various parameters used in the argument of \cite{Tr15} were not verified  properly.  In  \cite{ BMOR20}, Bennett \emph{et al.} fixed this problem for Dirichlet $L$-functions.

The objective of this article is to prove the following theorem.

\begin{theorem}\label{main-thm}
Given a number field $K$ of degree $n_K$ and with absolute discriminant $d_K$ and $r_1$ real places, for any $T\ge 1$, we have
\begin{equation}\label{main-bound-1}
\Big| N_K (T)  - \frac{T}{\pi} \log \Big( d_K \Big( \frac{T}{2\pi e}\Big)^{n_K}\Big)  + \frac{r_1}{4}\Big|
\le  0.22737  \log\Big(  \frac{d_K(T+2)^{n_K}}{(2\pi)^{n_K}}  \Big)  + 23.02528 n_K  +   4.51954.
\end{equation} 
\end{theorem}

In addition, writing the right of \eqref{main-bound-1} as $ C_1  \log \big(  \frac{d_K(T+2)^{n_K}}{(2\pi)^{n_K}}  \big)  +C_2  n_K  +C_3$, we have further admissible triples $(C_1,C_2,C_3)$ recorded in Table \ref{table2} in Section \ref{completingP}. Moreover, recalling that for $T\ge T_0$, $\log(T+2) -\log T \le \log (1+\frac{2}{T_0})$, from the above theorem and the triangle inequality, we derive the following improved bound for $N_K (T)$.

\begin{corollary}\label{main-coro}
Given a number field $K$ of degree $n_K$ and with absolute discriminant $d_K$, for any $T\ge 1$, we have
\begin{equation}\label{main-bound-2}
\Big| N_K (T)  - \frac{T}{\pi} \log \Big( d_K \Big( \frac{T}{2\pi e}\Big)^{n_K}\Big)\Big|
 \le  0.228 (\log d_K + n_K \log T) + 23.108  n_K  +4.520  .
\end{equation}
\end{corollary}

Furthermore, by Table \ref{table2}, writing the right of \eqref{main-bound-2} as $D_1(\log d_K + n_K \log T) + D_2 n_K  +D_3$, we have the following table of admissible $(D_1,D_2,D_3)$ that not only repair but also improve all  triples given in \cite[Table 2]{Tr15}.  (Note that, for all number fields $K$, our $D_2$ and $D_3$ yield a smaller vlaue of $D_2 n_K +D_3$ than the one given by Trudgian \cite{Tr15}.)

\begin{table}[htbp] 
\centering
\begin{tabular}{|c|c|c|c|c|c|c|c|c|c|    }
\hline
\multicolumn{5}{|c|}{  Trudgian \cite{Tr15} } & \multicolumn{5}{c|}{Our improvement } \\ 
 \hline
 & \multicolumn{2}{c|}{ $T\ge 1$} & \multicolumn{2}{c|}{ $T\ge 10$} & & \multicolumn{2}{c|}{ $T\ge 1$} & \multicolumn{2}{c|}{ $T\ge 10$} \\ 
 \hline
 $D_1$ & $D_2$ & $D_3$ & $D_2$ & $D_3$ & $D_1$ & $D_2$ & $D_3$ & $D_2$ & $D_3$  \\ 
 \hline
  0.247 & 8.851 & 3.024 & 8.726 &  2.081  & 0.245 & 6.735   & 4.213 & 6.449 & 3.124 \\ 
 \hline  
 0.265 & 7.521 &  3.178 & 7.396 & 2.101 & 0.264 & 5.276  & 4.082 & 4.968 & 3.051 \\ 
 \hline    
 0.282 & 6.776 & 3.335 & 6.651 & 2.123 &  0.281 & 4.478   & 4.010 & 4.149 & 3.012 \\ 
 \hline  
 0.299 & 6.262 & 3.494 & 6.138 & 2.146  & 0.296  & 3.971  & 3.969 & 3.622 & 2.990 \\ 
 \hline 
\end{tabular}
   \caption{Admissible $(D_1,D_2,D_3)$ in Corollary \ref{main-coro} and in \cite{Tr15}}\label{table1}
\end{table}

The proof of Theorem \ref{main-thm} follows closely  the arguments of  Bennett, Martin, O'Bryant, and Rechnitzer \cite{BMOR20}, Kadiri and Ng  \cite{KN12}, and Trudgian \cite{Tr15}, which are an adaption of the methods of Backlund \cite{Ba18},  McCurley  \cite{McC84}, and Rosser  \cite{Ro41}. We also take advantage of the refined estimates for Gamma factors obtained in \cite{ BMOR20}. Moreover, following the strategy of Bennett \emph{et al.} \cite{ BMOR20}, we extend  Rademacher's convexity bound for $\zeta_K(s)$ (cf. Propositions \ref{Rademacher} and \ref{convex-bd}) that, together with ``Backlund's trick'' (see Section \ref{B-trick-d}), plays a central role in improving the leading constants $C_1$ and $D_1$. Furthermore, we track all the parameters and related inequalities in a similar manner of Bennett \emph{et al.} \cite{ BMOR20} to fix the aforementioned error appearing in \cite{Tr15}. Last but not least, we note that  we obtain our results by a direct numerical computation (with help from Maple) and that it may be possible to use the ``interval analysis'' as in \cite{ BMOR20} to prove an estimate similar to \cite[Theorem 1.1]{ BMOR20}. Nonetheless, since Corollary \ref{main-coro} is already as strong as  \cite[Corollary 1.2]{ BMOR20}, and it is sufficient for most applications, we shall not devote ourselves to do such an interval analysis here.

\section{The main term and the gamma factor}\label{main}

\subsection{The main term}

Let $K$ be a number field of degree $n_K$ and with absolute discriminant $d_K$. We let $r_1$ and $r_2$ be the numbers of real and complex places, respectively, of $K$ and note that $n_K = r_1 +2r_2$. We define the completed zeta function $\xi_K(s) $ as
\begin{equation}\label{def-xi}
\xi_K(s)= s(s-1) d_K^{s/2} \g_K (s) \zeta_K (s),
\end{equation}
where 
$$
\gamma_K (s) = \Big(\pi^{-\frac{s+1}{2}} \Gamma \Big( \frac{s+1}{2}\Big) \Big)^{r_2} \Big( \pi^{-\frac{s}{2}} \Gamma \Big( \frac{s}{2}\Big)\Big)^{r_1 + r_2}.
$$
We recall that $\xi_K(s)$ extends to an entire function of order 1 and satisfies the functional equation
\begin{equation}\label{FE}
\xi_K(s) = \xi_K (1-s).
\end{equation}
As in the introduction, we set 
\[
N_K(T) = \# \{ \rho \in  \C  \mid  \zeta_K(\rho) =0,\  0  <\beta <  1, \ |\gamma| \leq T\}.
\]

To estimate $N_K(T)$, we shall apply the argument principle as follows. For any fixed $\sigma_1 >1$, we consider the rectangle $\mathcal{R}$ with vertices $\sigma_1 - iT,\  \sigma_1 + iT,\  1-  \sigma_1+ iT$, and $1- \sigma_1 - iT$ (that is away from zeros of $\xi_K (s)$).\footnote{Throughout our argument, we will always assume $T$ is away from zeros of $\xi_K(s)$. As shall be seen in Section \ref{completingP}, with this assumption, we will prove \eqref{final-est} for $T$ away from  zeros of $\xi_K(s)$. Nonetheless, if $T$ is the exact height of a zero, we know that $N_K(T)=N_K(T+\varepsilon)$ for all sufficiently small $\varepsilon>0$ (in other words, $T+\varepsilon$ is away from zeros). Then, by the triangle inequality, applying \eqref{final-est} with $T+\varepsilon$, we see that
\begin{align*}
&\Big| N_K (T)  - \frac{T}{\pi} \log \Big( d_K \Big( \frac{T}{2\pi e}\Big)^{n_K}\Big)  + \frac{r_1}{4}\Big|\\
&\le 
\Big| N_K (T+\varepsilon)  - \frac{T+\varepsilon}{\pi} \log \Big( d_K \Big( \frac{T+\varepsilon}{2\pi e}\Big)^{n_K}\Big)  + \frac{r_1}{4}\Big| 
+\Big| \frac{T+\varepsilon}{\pi} \log \Big( d_K \Big( \frac{T+\varepsilon}{2\pi e}\Big)^{n_K}\Big) -  \frac{T}{\pi} \log \Big( d_K \Big( \frac{T}{2\pi e}\Big)^{n_K}\Big) \Big| \\
&\le  C_1  \log\Big(  \frac{d_K(T+\varepsilon +2)^{n_K}}{(2\pi)^{n_K}}  \Big)  +C_2  n_K  +C_3
+\Big| \frac{T+\varepsilon}{\pi} \log \Big( d_K \Big( \frac{T+\varepsilon}{2\pi e}\Big)^{n_K}\Big) -  \frac{T}{\pi} \log \Big( d_K \Big( \frac{T}{2\pi e}\Big)^{n_K}\Big) \Big|.
\end{align*}
Now, taking $\varepsilon\rightarrow 0^+$, we conclude that \eqref{final-est} is also valid when  $T$ is the exact height of a zero.
} 
As $\xi_K (s)$ is entire,  it  follows  from the  argument principle that 
\[
N_K(T) = \frac{1}{2\pi } \Delta_{\mathcal{R}} \arg \xi_K (s).
\]
Let $\mathcal{C}$ be the part of the contour of $\mathcal{R}$ in $\Re (s ) \geq \frac{1}{2}$ and $\mathcal{C}_0$ be the part of the contour of $\mathcal{R}$ in $\Re (s ) \geq \frac{1}{2}$ and $\Im (s) \geq 0$. Since $\overline{\xi_K(s)} = \xi_K (\bar{s})$,  the functional equation \eqref{FE} then yields
\[
\Delta_{\mathcal{R}} \arg \xi_K (s) = 2 \Delta_{\mathcal{C}} \arg \xi_K (s)  =4    \Delta_{\mathcal{C}_0} \arg \xi_K (s),
\]
which implies that
\begin{align}\label{formula-N-K}
N_K(T)    = \frac{2}{\pi}   \Delta_{\mathcal{C}_0} \arg \xi_K (s).
\end{align}
Writing $B= d_K /\pi^{n_K} $, by \eqref{def-xi}, we have
\begin{align}\label{deta-expansion}
 \begin{split}
\Delta_{\mathcal{C}_0} \arg \xi_K(s) 
&= \Delta_{\mathcal{C}_0} \arg s + \Delta_{\mathcal{C}_0} \arg B^{s/2} \\
&+(r_1 + r_2)  \Delta_{\mathcal{C}_0} \arg \Gamma  \left(\frac{s}{2} \right)  +  r_2 \Delta_{\mathcal{C}_0} \arg \Gamma \left(\frac{s+1}{2} \right) + \Delta_{\mathcal{C}_0} \arg \left( (s-1) \zeta_K (s)\right).
 \end{split}
\end{align}
It is clear that
\begin{align}   \label{deta-explicit}
 \begin{split}
&\Delta_{\mathcal{C}_0} \arg s = \arctan (2T),\\
&\Delta_{\mathcal{C}_0} \arg B^{s/2} = \frac{T}{2} \log B = \frac{T}{2} \log \Big( \frac{d_K}{\pi^{n_K} }\Big),\\
&\Delta_{\mathcal{C}_0}  \arg \Gamma (s) = \Delta_{\mathcal{C}_0}  (\Im \log \Gamma (s)) = \Im \log \Gamma \Big(\frac{1}{2} + iT \Big) .
 \end{split}
\end{align}

  To control the Gamma factor, we shall appeal for the improved numerical bound established in \cite[Sec. 3]{BMOR20}. For $a\in\{0,1\}$, we set
$$
g_a(T)= \frac{2}{\pi}\Im \log \Gamma \Big(\frac{1}{4} +\frac{a}{2} + i\frac{T}{2} \Big) - \frac{T}{\pi} \log \Big( \frac{T}{2e}\Big) -\frac{2a -1}{4}.
$$
It follows from  \cite[Proposition 3.2]{BMOR20} that for $a\in\{0,1\}$ and $T\ge 5/7$, 
$$
|g_a(T)|\le \frac{2-a}{50T}.
$$
Hence, setting
\begin{align}\label{equ:g-K}
g_K(T) = (r_1 +r_2 )g_0(T) +  r_2 g_1(T), 
\end{align}
we then obtain
\begin{align}\label{gK-bd}
|g_K(T) | \le \frac{2 n_K}{50T} -\frac{r_2}{50T}.
\end{align}

Now, gathering  \eqref{formula-N-K}, \eqref{deta-expansion}, \eqref{deta-explicit}, and \eqref{equ:g-K}, we obtain
\begin{equation}\label{bd-N-K-1}
N_K (T) = \frac{2}{\pi}\arctan ( 2T) + g_K(T) +  \frac{T}{\pi} \log \Big( d_K \Big( \frac{T}{2\pi e}\Big) ^{n_K} \Big)- \frac{r_1}{4} 
+ \frac{2}{\pi}\Delta_{\mathcal{C}_0} \arg ( (s-1) \zeta_K (s) ).
\end{equation}
Let $\mathcal{C}_1$ denote the vertical line from $\sigma_1$ to $\sigma_1 + iT$ and $\mathcal{C}_2$  denote the horizontal line from $\sigma_1 + iT$ to $\frac{1}{2} + iT$. We require the following two estimates.

\begin{lemma} \label{bd-zeta-K}
For $s=\sigma+it$ with $\sigma>1$, one has
\begin{align*}
\frac{\zeta_K (2 \sigma)}{\zeta_K (\sigma)}  \leq |\zeta_K (s)| \leq \zeta (\sigma)^{n_K},
\end{align*}
where, as later, $\zeta(s)$ denotes the Riemann zeta function.
\end{lemma}

\begin{lemma} \label{lm-Delta-C1}
For $\sigma_1 > 1$,
\[
|\Delta_{\mathcal{C}_1} \arg (s-1)\zeta_K (s)| \leq \frac{\pi}{2} + n_K \log \zeta(\sigma_1).
\]
\end{lemma}
\begin{proof}
Note that
\begin{align*}
\Delta_{\mathcal{C}_1} \arg (s-1)\zeta_K (s)
 &=\Delta_{\mathcal{C}_1}  \arg (s-1) +\Delta_{\mathcal{C}_1}  \arg \zeta_K (s) 
= \arctan \Big( \frac{T}{\sigma_1 -1}\Big) +\Delta_{\mathcal{C}_1} \arg \zeta_K (s).
\end{align*}
Now, the lemma follows from the estimate
\[
|\Delta_{\mathcal{C}_1}  \arg  \zeta_K (s)|
 =  |\arg \zeta_K (\sigma_1+iT) |
  \leq |\log \zeta_K (\sigma_1+iT) |
  \le  \log \zeta_K (\sigma_1)  \leq n_K \log \zeta(\sigma_1),
\]
where the last inequality is due to Lemma \ref{bd-zeta-K}.
\end{proof}
Thus, by Lemma \ref{lm-Delta-C1} and \eqref{bd-N-K-1}, we arrive at 
\begin{equation}\label{main-est}
 \Big| N_K (T)  -  \frac{T}{\pi} \log \Big( d_K\Big( \frac{T}{2\pi e}\Big) ^{n_K} \Big) +  \frac{r_1}{4} \Big|  \leq  
 2 + |g_K(T)|+ \frac{2n_K}{\pi} \log \zeta(\sigma_1) +  \frac{2}{\pi}| \Delta_{\mathcal{C}_2}  \arg ( (s-1) \zeta_K (s) )|.
\end{equation}

\subsection{Bounding the Gamma factor}

For $a\in\{0,1\}$, $0\le d < 9/2$ and $T\ge 5/7$, we set
$$
\mathcal{E}_a(T,d)=  \Big|  \Im \log \Gamma \Big( \frac{\sigma +a+iT}{2} \Big)\Big|_{\sigma=\frac{1}{2}}^{\frac{1}{2}+d} 
+  \Im \log \Gamma \Big( \frac{\sigma +a+iT}{2} \Big)\Big|_{\sigma=\frac{1}{2}}^{\frac{1}{2}-d} \Big|,
$$
and we define
\begin{equation}\label{def-EK-1}
\mathcal{E}_K(T,d)=  (r_1+r_2)\mathcal{E}_0(T,d) +  r_2 \mathcal{E}_1(T,d). 
\end{equation}
Following \cite[p. 1463]{BMOR20}, we let
\begin{align*}
 \begin{split}
E_a (T,d)& = \frac{2T /3}{(2a+2d + 17)^2 + 4 T^2} + \frac{2T/3}{(2a -2d +17)^2 + 4T^2}  - \frac{4T/3}{(2a+17)^2 + 4T^2}\\
& + \frac{T}{2} \log \Big( 1 + \frac{(2a + 17)^2}{4T^2}\Big) - \frac{T}{4} \log \Big( 1 + \frac{(2a + 2d +17)^2}{4T^2}\Big) - \frac{T}{4} \log \Big( 1 + \frac{(2a -2d + 17)^2}{4T^2}\Big)\\
&+ \frac{(8+6\pi)/45}{((2a + 2d + 17)^2 + 4T^2)^{3/2}}
 +  \frac{(8+6\pi)/45}{((2a - 2d + 17)^2 + 4T^2)^{3/2}} + \frac{2(8+ 6\pi)/45}{((2a + 17)^2 + 4T^2)^{3/2}} \\
& + \sum_{k=0}^3\Big( 2\arctan \frac{2a+1+4k}{2T} - \arctan \frac{2a+2d+1+4k}{2T} - \arctan \frac{2a-2d+1+4k}{2T} \Big)\\  
& + \frac{2a + 2d +15}{4} \arctan \frac{2a + 2d + 17}{2T}  + \frac{2a - 2d +15}{4} \arctan \frac{2a -2d + 17}{2T}\\
& - \frac{2a + 15}{2} \arctan\frac{2a+17}{2T}.
 \end{split} 
\end{align*}
We shall further set
\begin{equation}\label{def-EK-2}
E_K(T,d) = (r_1+r_2)E_0(T,d) +  r_2 E_1(T,d). 
\end{equation}
As shown in \cite[p. 1462]{BMOR20}, $\mathcal{E}_a(T,d) \le E_a (T,d)$ for $0 \leq d < 9/2$ and  $T\ge 5/7$, and thus
\begin{equation}\label{EK-bd}
\mathcal{E}_K(T,d) \le E_K (T,d)
\end{equation}
for $0 \leq d < 9/2$ and  $T\ge 5/7$. In addition,  from \cite[Lemma 3.4]{BMOR20} and our definition of $E_K(T,d)$, we have the following lemma.

\begin{lemma}\label{EK-bd-final}
For $0\le \delta_1 \le d < 9/2$ and  $T\ge 5/7$, 
$$
0<E_K (T,\delta_1)\le E_K(T,d).
$$
Furthermore, for $d\in [\frac{1}{4}, \frac{5}{8}] $ and $T\ge 5/7$,
$$
\frac{E_K (T,d)}{\pi} \le (r_1+ r_2)\frac{640d-112}{1536(3T-1)} +r_2\frac{(640+216)d-112-39}{1536(3T+3-1)} + \frac{n_K}{2^{10}}.
$$
\end{lemma}

\section{Backlund's trick and the Jensen integral}\label{B-trick}

\subsection{Introducing the auxiliary function $f_N$}

For the sake of convenience, we shall set $\mathcal{Z}(w)=(w-1)\zeta_K(w)$. In order to analyse the variation of the argument of $\mathcal{Z}(w)$ on $\mathcal{C}_2$, we shall introduce an auxiliary function
\begin{align*}
f_N(s)=\frac{1}{2}\Big(\mathcal{Z}(s+iT)^N+\mathcal{Z}(s-iT)^N\Big)
\end{align*}
for  $N\in\Bbb{N}$. For $\sigma\in\R$, it is clear that 
\begin{align*}
f_N(\sigma)=\frac{1}{2}\Big(\mathcal{Z}(\sigma+iT)^N+\mathcal{Z}(\sigma-iT)^N\Big)=\frac{1}{2}\Big(\mathcal{Z}(\sigma+iT)^N+\overline{\mathcal{Z}(\sigma+iT)^N}\Big)
=\mathfrak{Re}(\mathcal{Z}(\sigma+iT)^N).
\end{align*}
We need the following definition that measures the variation of the argument of $\mathcal{Z}(w)^N$ on $\mathcal{C}_2$.

\begin{defn}\label{defn}
Let $b_N$ denote the non-negative integer, depending on $N$, such that
\begin{align*}
b_N\leq\frac{1}{\pi}\Big|\Delta_{\mathcal{C}_2} \arg\mathcal{Z}(w)^N\Big|<b_N+1.
\end{align*}
\end{defn}

From this definition and the fact that $\arg\mathcal{Z}(w)^N=N\arg\mathcal{Z}(w)$, we immediately obtain
\begin{align}
\frac{b_N}{N}
\leq\frac{1}{\pi}\Big|\Delta_{\mathcal{C}_2} \arg\mathcal{Z}(w)\Big|
<\frac{b_N+1}{N}.
\label{equ:bd-n over N}
\end{align}
In addition, we have the following lemma concerning the zeros of  $f_N(\sigma)$.

\begin{lemma}\label{zeros-f}
In the notation of Definition \ref{defn}, the function $f_N(\sigma)$ has at least $b_N$ zeros in  $[\frac{1}{2},\sigma_1]$.
\end{lemma}
\begin{proof}
By Definition \ref{defn}, there are at least $b_N$ different values of $\sigma$ such that
$\frac{1}{2}+\frac{1}{\pi}\arg\mathcal{Z}(\sigma+iT)^N \in\Bbb{Z}$. Thus, for such values of $\sigma$, 
$\mathcal{Z}(\sigma+iT)^N$ is purely imaginary, which means that
\begin{align*}
f_N(\sigma)=\mathfrak{Re}(\mathcal{Z}(\sigma+iT)^N)=0
\end{align*}
for at least $b_N$ different values $\sigma$.
\end{proof}

We shall also require the following lemma regarding the limiting behaviour of $f_N$.

\begin{lemma}\label{Lem-4.3}
For any $c>1$, there is an infinite sequence of natural numbers $(N_m)_{m=1}^{\infty}$ such that $f_{N_m}(c)\neq 0$.  Moreover, we have
\begin{align*}
\limsup_{m\to\infty}\Big(-\frac{1}{N_m}\log |f_{N_m}(c)|\Big)
\leq\log\Big(\frac{1}{\sqrt{(c-1)^2+T^2}}\frac{\zeta_K(c)}{\zeta_K(2c)}\Big) .
\end{align*}
\end{lemma}
\begin{proof}
Write  $\mathcal{Z}(c+iT)=Re^{i\phi}$ for some $R,\phi\in\Bbb{R}$. It is clear that $\mathcal{Z}(c-iT)=Re^{-i\phi}$. Also, as $\mathcal{Z}(c+iT)\neq 0$ for any $c>1$, we know that $R>0$. Thus, we have
\begin{align*}
\frac{f_N(c)}{\mathcal{Z}(c+iT)^{N}}=\frac{1}{2}\Big(1+\frac{\mathcal{Z}(c-iT)^N}{\mathcal{Z}(c+iT)^N}\Big)=\frac{1}{2}(1+e^{-2N\phi i})
\end{align*}
for any $N\in\Bbb{N}$.

Now, applying Dirichlet's approximation theorem, for any $\phi$, there is an infinite sequence of natural numbers $(N_m)_{m=1}^{\infty}$  such that as $m\rightarrow \infty$, $-2N_m\phi\rightarrow 0$ modulo $2\pi$ and $N_m \rightarrow \infty$. Thus,  $\frac{f_{N_m}(c)}{\mathcal{Z}(c+iT)^{N_m}}\to 1$ as $m\rightarrow \infty$, and hence
$$
\lim_{m\rightarrow \infty}\Big(-\frac{1}{N_m} (\log |f_{N_m}(c)| -N_m \log |\mathcal{Z}(c+iT)|) \Big)
= \Big(\lim_{m\rightarrow \infty}  \frac{-1}{N_m}\Big) \Big( \lim_{m\rightarrow \infty} \log \Big|\frac{f_{N_m}(c)}{\mathcal{Z}(c+iT)^{N_m}} \Big| \Big)=0.
$$
Moreover, by the left inequality of Lemma \ref{bd-zeta-K}, we have
\begin{align*}
|\mathcal{Z}(c+iT)|\geq\sqrt{(c-1)^2+T^2}\frac{\zeta_K(2c)}{\zeta_K(c)},
\end{align*}
which, combined with the above identity, gives
\begin{align*}
0&\ge \limsup_{m\to\infty} \Big( -\frac{1}{N_m}\log |f_{N_m}(c)|+ \log\Big(\sqrt{(c-1)^2+T^2}\frac{\zeta_K(2c)}{\zeta_K(c)}\Big)  \Big) \\
&=  \limsup_{m\to\infty} \Big( -\frac{1}{N_m}\log |f_{N_m}(c)| \Big) + \log\Big(\sqrt{(c-1)^2+T^2}\frac{\zeta_K(2c)}{\zeta_K(c)}\Big) .
\end{align*}
Herein, we complete the proof.
\end{proof}

Let $D(c,r)$ be the open disk centred at $c$ with radius $r$. Let $(N_m)_{m=1}^{\infty}$  be given as in Lemma \ref{Lem-4.3}. For any $N\in  (N_m)_{m=1}^{\infty}$, we set
\begin{align*}
S_N(c,r)=\frac{1}{N}\sum_{z\in\mathcal{S}_N(D(c,r) )  }\log\frac{r}{|z-c|},
\end{align*}
where $\mathcal{S}_N(D(c,r) )$ denotes the set of  zeros of $f_N(s)$ in $D(c,r)$. As in \cite[Theorem 5.1]{BMOR20}, we have the following version of  Jensen's formula.

\begin{theorem}[Jensen's formula]
For $c\in\mathbb{C}$ and $r>0$, if $f_N(c)\neq0$, then
\begin{align*}
S_N(c,r)=-\frac{1}{N}\log|f_N(c)|+\frac{1}{2\pi}\int_{-\pi}^{\pi}\frac{1}{N}\log|f_N(c+re^{i\theta})|d\theta .
\end{align*}
\end{theorem}

Applying Jensen's formula and Lemma \ref{Lem-4.3}, we obtain the following  upper bound for $S_N(c,r)$.

\begin{proposition}\label{Je-bd}
Let $c$, $r$, and $\sigma_1$ be real numbers such that
\begin{align*}
c-r<\frac{1}{2}<1<c<\sigma_1<c+r.
\end{align*}
Let  $F_{c,r}:[-\pi,\pi]\to\mathbb{R}$ be an even function such that $F_{c,r}(\theta)\geq\frac{1}{N_m}\log|f_{N_m}(c+re^{i\theta})|$. Then we have
\begin{align*}
\limsup_{m\to\infty}S_{N_m}(c,r)\leq\log\Big(\frac{1}{\sqrt{(c-1)^2+T^2}}\frac{\zeta_K(c)}{\zeta_K(2c)}\Big)+\frac{1}{\pi}\int_{0}^{\pi}F_{c,r}(\theta)d\theta.
\end{align*}
\end{proposition}

\subsection{Backlund's trick}\label{B-trick-d}

We start with the following technical estimate.

\begin{lemma}  \label{bd-arg-zeta}
Let $0\le d < 1/2$ and $T\ge 5/7$. Then we have
\begin{align*}
\Big| \arg \Big( (\sigma  -1+ i T )\zeta_K(\sigma  + i T) \Big)^N\Big|_{\sigma=\frac{1}{2}}^{\frac{1}{2}+d} \Big|
& \leq \Big|\arg \Big( (\sigma -1 + i T)\zeta_K(\sigma  +i T) \Big)^N\Big|_{\sigma=\frac{1}{2}}^{\frac{1}{2}-d} \Big|\\
& + N \mathcal{E}_K (T,d) + N \frac{\pi}{2},
\end{align*}
where $\mathcal{E}_K (T,d)$ is defined as in \eqref{def-EK-1}.
\end{lemma}

\begin{proof}
By the  functional equation \eqref{FE} and the fact that $\xi_K(s)=\overline{\xi_K(\bar{s})}$, we have
\begin{align}\label{lem-bd-arg-zeta-1}
 \arg \xi_K (\sigma  + i T) \Big|_{\sigma=\frac{1}{2}}^{\frac{1}{2}+d}
 =  - \arg  \xi_K (\sigma + i T) \Big|_{\sigma=\frac{1}{2}}^{\frac{1}{2}-d} .
\end{align}
Since
$$
\arg (\sigma + i T) + \arg B^{(\sigma + i T)/2}=  \arctan \frac{T}{\sigma} + \frac{T}{2} \log B, 
$$
by \eqref{def-xi}, we have
\begin{align}\label{expression-arg-xi}
 \begin{split}
\arg \xi_K (\sigma + i T)   
&= \arctan \frac{T}{\sigma} + \frac{T}{2} \log B +(r_1 + r_2)  \Im \log \Gamma  \Big(\frac{\sigma + i T}{2} \Big)  +  r_2 \Im \log \Gamma \Big(\frac{\sigma + i T+1}{2} \Big) \\
& +\arg \Big( (\sigma + i T-1) \zeta_K (\sigma + i T)\Big).
 \end{split}
\end{align}
As we know that for $\pm xy<1$, 
$$
\arctan x \pm \arctan y = \arctan \frac{x\pm y}{1\mp xy},
$$
for $0 \le d < 1/2$, we have
\begin{align}\label{arctan-bd}
 \begin{split}
&\Big| \arctan \frac{T}{\frac{1}{2}+d} - \arctan \frac{T}{\frac{1}{2}} + \arctan \frac{T}{\frac{1}{2}-d} - \arctan \frac{T}{\frac{1}{2}}\Big|\\
&= \Big| \arctan \frac{ \frac{T}{\frac{1}{2}+d} -\frac{T}{\frac{1}{2}} }{1+ \frac{T}{\frac{1}{2}+d}\frac{T}{\frac{1}{2}} } 
+\arctan \frac{ \frac{T}{\frac{1}{2}-d} -\frac{T}{\frac{1}{2}} }{1+ \frac{T}{\frac{1}{2}-d}\frac{T}{\frac{1}{2}} }\Big|\\
&\le \frac{\pi}{2}.
 \end{split}
\end{align}
Now, applying  the triangle inequality, by  \eqref{lem-bd-arg-zeta-1},  \eqref{expression-arg-xi}, and \eqref{arctan-bd}, we obtain
\begin{align*}
\Big| \arg \Big( (\sigma  -1+ i T )\zeta_K(\sigma  + i T) \Big)\Big|_{\sigma=\frac{1}{2}}^{\frac{1}{2}+d} \Big|
&\le\Big| \arg \Big( (\sigma  -1+ i T )\zeta_K(\sigma  + i T) \Big)\Big|_{\sigma=\frac{1}{2}}^{\frac{1}{2}-d} \Big|    +\mathcal{E}_K(T,d)+  \frac{\pi}{2}.
\end{align*}
Recalling that 
$$\arg \Big( (\sigma  -1+ i T )\zeta_K(\sigma  + i T) \Big)^N\Big|_{\sigma=\frac{1}{2}}^{\frac{1}{2}\pm d}
  =N \arg \Big( (\sigma  -1+ i T )\zeta_K(\sigma  + i T) \Big)\Big|_{\sigma=\frac{1}{2}}^{\frac{1}{2}\pm d}, 
$$
we conclude the proof.
\end{proof}

As argued in  \cite{BMOR20} and \cite{Tr15}, we require the following version of ``Backlund's trick''.

\begin{proposition}[Backlund's trick]\label{BT}
Let $c$ and $r$ be real numbers. Set 
\[
\sigma_1 = c + \frac{(c-1/2)^2}{r} \quad \text{and} \quad  \delta = 2c - \sigma_1 -\frac{1}{2}.
\]
If $1 < c < r$ and $0 < \delta < \frac{1}{2}$, then
\[ 
\Big|  \arg \Big( (\sigma + iT-1)\zeta_K(\sigma + iT) \Big) \Big|_{\sigma = \sigma_1}^{1/2} \Big| \leq \frac{\pi S_N(c,r)}{2 \log (r/ (c-1/2))} + \frac{E_K(T,\delta)}{2} +\frac{\pi}{N}+ \frac{\pi}{2N}+
\frac{\pi}{4}.
\]
\end{proposition}

\begin{proof}
By the conditions on $c$ and $r$ and the definitions of $\sigma_1$ and $\delta$, we know that
\[
c -r < \frac{1}{2} - \delta \leq \frac{1}{2} \leq \frac{1}{2} + \delta = 2c-\sigma_1 \leq c \leq \sigma_1 < c+r.
\]
As  $\log \frac{r}{|z-c|} >0$ for $z \in D(c,r)$, we see that 
\[
S_N(c,r) = \frac{1}{N} \sum_{z \in \mathcal{S}_N (D (c,r))} \log \frac{r}{|z-c|} \geq \frac{1}{N }\sum_{z \in \mathcal{S}_N ((c-r,\sigma_1])} \log \frac{r}{|z-c|}.
\]
Recall that by Lemma \ref{zeros-f},  there are at least $b_N$ values of $\sigma$   satisfying $\sigma \in [1/2,\sigma_1]$ and $f_N(\sigma) =0$, where $b_N$ is defined as in Definition \ref{defn}.  For $1 \leq k \leq b_N$, we then set $\delta_k$ as the smallest non-negative  real number such that 
\begin{align}\label{cond-k}
f_N (1/2+ \delta_k) =0 \quad \text{and} \quad k-1 \leq \frac{1}{\pi} \Big| \arg \Big( (\sigma + iT-1)\zeta_K(\sigma + iT) \Big)^N \Big|_{\sigma = 1/2}^{1/2 + \delta_k}  \Big|.
\end{align}
Writing $z_k = \frac{1}{2} + \delta_k$, we let $x_1$ denote the number of $z_k$ with $z_k \in [1/2, 1/2 + \delta) = [ 1/2, 2c- \sigma_1)$ and  let $x_2 $ denote the number of $z_k$ with $z_k \in [2c- \sigma_1, \sigma_1]$. We note that $x_2 = b_N -x_1$ and that
\[
0 \leq \delta_1 < \delta_2  < \cdots< \delta_{x_1} < \delta \leq  \delta_{x_1 +1} < \cdots < \delta_{b_N} \leq \sigma_1 -1/2.
\]
From  \eqref{EK-bd}, \eqref{cond-k}, and Lemma  \ref{bd-arg-zeta},  it follows that 
\begin{align*}
k-1&\le  \frac{1}{\pi}\Big|\arg \Big( (\sigma -1 + i T)\zeta_K(\sigma  +i T) \Big)^N\Big|_{\sigma=\frac{1}{2}}^{\frac{1}{2}-\delta_k} \Big| + \frac{1}{\pi}N E_K (T,\delta_k) + \frac{N}{2}
\end{align*}
whenever $1 \leq k \leq x_1$ (which implies that $\delta_k< \delta <\frac{1}{2}$).

For each $j\geq 1$, if there exists a $k$ (chosen to be minimal) such that 
\begin{equation*}
k - 1-  \frac{1}{\pi} N E_K(T, \delta_k) - \frac{N}{2} \geq j,
\end{equation*}
then $f_N$ has at least $j$ zeros in $[ 1/2- \delta_k, 1/2)$ since
\[
\frac{1}{\pi} \Big| \arg \Big( (\sigma + iT-1)\zeta_K(\sigma + iT) \Big)^N \Big|_{\sigma = 1/2}^{1/2 - \delta_k}  \Big|   \ge k-1 -  \frac{1}{\pi} N E_K(T, \delta_k) - \frac{N}{2} \geq j.
\]
For such an instance, we define $\delta_{-k}$ as the smallest values of these zeros (to avoid possible repetition), and we shall say that the zero $z_k =1/2 + \delta_k$ has a pair $z_{-k} = 1/2 - \delta_{-k}$. We note that $\delta_{-k} \leq \delta_k$ by the construction.

By the same argument as in  \cite[pp. 1467-1468]{BMOR20}, we have
\[
S_{N}(c,r) \geq \frac{2b_N - \frac{NE_K(T,\delta) + \frac{N\pi}{2} + \pi}{\pi}}{N} \log \Big( \frac{r}{c-1/2}\Big), 
\]
and thus
\[
\frac{b_N}{N} \leq \frac{S_N(c,r)}{2 \log (r/ (c-1/2))} + \frac{E_K(T,\delta)}{2\pi} +\frac{1}{4}+\frac{1}{2N},
\]
which combined with \eqref{equ:bd-n over N} completes the proof. 
\end{proof}

\subsection{Constructing  and bounding $F_{c,r} $}

We first recall the convexity bound for $\zeta_K(s)$ established by Rademacher \cite[Theorem 4]{Ra59}.
\begin{proposition}\label{Rademacher}
Let $\eta\in(0,\frac{1}{2}]$ and $s=\sigma+it$.  If $-\eta\leq\sigma\leq1+\eta$, then one has
\begin{align*}
|\zeta_K(s)|
\leq 3\Big|\frac{1+s}{1-s}\Big|\Big(d_K\Big(\frac{|1+s|}{2\pi}\Big)^{n_K}\Big)^{\frac{1+\eta-\sigma}{2}}\zeta(1+\eta)^{n_K}.
\end{align*}
Also, for $\sigma \in [-\frac{1}{2},0)$, one has
\begin{align}\label{Ra-conv-bd}
|\zeta_K(s)|\leq3\Big|\frac{1+s}{1-s}\Big|\Big(d_K\Big(\frac{|1+s|}{2\pi}\Big)^{n_K}\Big)^{\frac{1}{2}-\sigma}\zeta(1-\sigma)^{n_K}.
\end{align}
\end{proposition}

We note that the second inequality follows from the first bound by taking $\eta=-\sigma$. Moreover,  Rademacher's argument \cite{Ra59} can be used to extend \eqref{Ra-conv-bd} for $\sigma <0$ as follows (cf. \cite[Theorem 5.7]{BMOR20}). For $x\in\mathbb{R}$, let $[x]$ be the integer closest to $x$; when there are two integers equally close to $x$, we shall  choose the one closer to 0.

\begin{proposition}\label{convex-bd}
Let $s=\sigma+it$ with $\sigma<0$. Then we have
\begin{align*}
|\zeta_K(s)|\leq\Big(\frac{d_K}{(2\pi)^{n_K}}\Big)^{\frac{1}{2}-\sigma}|1+s-[\sigma]|^{n_K(\frac{1}{2}+[\sigma]-\sigma)}\prod_{j=1}^{-[\sigma]} | s+j-1|^{n_K}\zeta(1-\sigma)^{n_K}.
\end{align*}
\end{proposition}

\begin{proof}
From the functional equation \eqref{FE} we have
\begin{align*}
|\zeta_K(s)|&\leq d_K^{1/2-\sigma}\Big|\frac{\gamma_K(1-s)}{\gamma_K(s)}\Big||\zeta_K(1-s)|\\
&=d_K^{1/2-\sigma}\pi^{(\sigma-\frac{1}{2})n_K}\Big|\frac{\Gamma(\frac{1}{2}+\frac{1-s}{2})}{\Gamma(\frac{1}{2}+\frac{s}{2})}\Big|^{r_2}\Big|\frac{\Gamma(\frac{1-s}{2})}{\Gamma(\frac{s}{2})}\Big|^{r_1+r_2}|\zeta_K(1-s)|.
\end{align*}

As $\sigma<0$, by Lemma \ref{bd-zeta-K}, we have $|\zeta_K(1-s)|\leq{\zeta(1-\sigma)}^{n_K}$. It remains to estimate the ratios of gamma functions. It was obtained in the proof of \cite[Theorem 5.7]{BMOR20} that for $a,b \in \{0,1 \}$ and $k\in\Bbb{Z}$,
\begin{align*}
\frac{\Gamma(\frac{a}{2}+\frac{1-s}{2})}{\Gamma(\frac{a}{2}+\frac{s}{2})}
=\frac{\Gamma(\frac{b}{2}+\frac{1-(s+k)}{2})}{\Gamma(\frac{b}{2}+\frac{s+k}{2})}2^{-k}\Big(\prod_{j=1}^{k} (s+j-1)\Big)\frac{\sin(\frac{\pi}{2}(s+k+1-b))}{\sin(\frac{\pi}{2}(s+1-a))}.
\end{align*}

Setting $a=0$ and $a=1$ and taking $b\equiv k\  (\modd 2) $ and $b\equiv k+1\ (\modd 2)$, respectively, we can make sine factors $\pm1$. Thus, upon choosing $k=-[\sigma]$ and applying \cite[Lemmata 1 and 2]{Ra59} to $\frac{\Gamma(\frac{b}{2}+\frac{1-(s+k)}{2})}{\Gamma(\frac{b}{2}+\frac{s+k}{2})}$, we conclude that
\begin{align*}
\Big|\frac{\Gamma(\frac{1-s}{2})}{\Gamma(\frac{s}{2})}\Big|^{r_1+r_2}\leq\Big(\frac{1}{2}|1+s-[\sigma]|\Big)^{(\frac{1}{2}+[\sigma]-\sigma)(r_1+r_2)}2^{[\sigma](r_1+r_2)}\Big(\prod_{j=1}^{-[\sigma]} | s+j-1 |\Big)^{r_1+r_2}
\end{align*}
and 
\begin{align*}
\Big|\frac{\Gamma(\frac{1}{2}+\frac{1-s}{2})}{\Gamma(\frac{1}{2}+\frac{s}{2})}\Big|^{r_2}\leq\Big(\frac{1}{2}|1+s-[\sigma]|\Big)^{(\frac{1}{2}+[\sigma]-\sigma)r_2}2^{[\sigma]r_2}\Big(\prod_{j=1}^{-[\sigma]} |s+j-1|\Big)^{r_2}.
\end{align*}
Collecting above estimates and recalling the fact that $n_K=r_1+2r_2$, we obtain the desired result.
\end{proof}

\begin{lemma}
Let $\eta\in(0,\frac{1}{2}]$, $s=\sigma+it$, and $T>0$. If $\sigma\geq1+\eta$, then we have
\begin{align*}
\frac{1}{N}\log|f_N(s)|\leq\frac{1}{2}\log((\sigma-1)^2+(|t|+T)^2)+n_K\log\zeta(\sigma).
\end{align*}
If $-\eta\leq\sigma\leq1+\eta$, then we have
\begin{align*}
\frac{1}{N}\log|f_N(s)|&\leq\log3+\frac{n_K(1+\eta-\sigma)+2}{4}\log((\sigma+1)^2+(|t|+T)^2)\\
&+\frac{1+\eta-\sigma}{2}\log\Big(\frac{d_K}{(2\pi)^{n_K}}\Big) +n_K\log\zeta(1+\eta).
\end{align*}
If $\sigma\leq-\eta$, then we have
\begin{align*}
\frac{1}{N}\log|f_N(s)|&\leq n_K\log\zeta(1-\sigma)+\frac{1}{2}\log((\sigma-1)^2+(|t|+T)^2)\\
&+\frac{1-2\sigma}{2}\log\Big(\frac{d_K}{(2\pi)^{n_K}}\Big)+\frac{(1-2\sigma+2[\sigma])n_K}{4}\log((1+\sigma-[\sigma])^2+(|t|+T)^2)\\
&+\frac{n_K}{2}\sum_{j=1}^{-[\sigma]} \log((\sigma+j-1)^2+(|t|+T)^2) .
\end{align*}
\end{lemma}

\begin{proof}
Since $\sigma\geq1+\eta>1$, by Lemma \ref{bd-zeta-K}, we derive
\begin{align*}
|f_N(s)|&\leq\frac{1}{2}\Big(|s+iT-1|^N|\zeta_K(s+iT)|^N+|s-iT-1|^N|\zeta_K(s-iT)|^N\Big)\\
&\leq\Big((\sigma-1)^2+(|t|+T)^2\Big)^{\frac{N}{2}}\zeta(\sigma)^{n_K{N}}.
\end{align*}
Now, the first estimate follows from taking logarithms and dividing both sides by $N$.

Secondly, if $-\eta\leq\sigma\leq1+\eta$, then by Proposition \ref{Rademacher}, we see that $|f_N(s)|$ is at most
\begin{align*}
&\frac{1}{2}\Big(3^N|s+iT+1|^N + 3^N|s-iT+1|^N\Big)    \Big(d_K\Big(\frac{\sqrt{(\sigma+1)^2+(|t|+T)^2}}{2\pi}\Big)^{n_K}\Big)^{\frac{(1+\eta-\sigma)N}{2}}\zeta(1+\eta)^{n_K{N}}\\
&\leq 3^N\Big((\sigma+1)^2+(|t|+T)^2\Big)^{\frac{N}{2}}\Big(d_K\Big(\frac{\sqrt{(\sigma+1)^2+(|t|+T)^2}}{2\pi}\Big)^{n_K}\Big)^{\frac{(1+\eta-\sigma)N}{2}}\zeta(1+\eta)^{n_K{N}}.
\end{align*}
Again, taking logarithms yields the second bound.

Lastly, for $\sigma\leq-\eta$, it follows from Proposition \ref{convex-bd} that
\begin{align*}
|f_N(s)|
&\leq\Big((\sigma-1)^2+(|t|+T)^2\Big)^{\frac{N}{2}}\Big(\frac{d_K}{(2\pi)^{n_K}}\Big)^{N(\frac{1}{2}-\sigma)}|(1+\sigma-[\sigma])^2+(|t|+T)^2|^{\frac{(1-2\sigma+2[\sigma])Nn_K}{4}}\\
&\times\Big(\prod_{j=1}^{-[\sigma]}((\sigma+j-1)^2+(|t|+T)^2)\Big)^{\frac{n_K{N}}{2}}\zeta(1-\sigma)^{n_K{N}}.
\end{align*}
We then conclude the proof by taking logarithms.
\end{proof}

Following \cite{BMOR20}, to proceed further, we introduce some notation and auxiliary functions. We first set
\begin{align*}
L_j(\theta)=\log\frac{(j+c+r\cos\theta)^2+(|r\sin\theta|+T)^2}{(T+2)^2}.
\end{align*}
and note that $L_j(\theta)$ is an even function of $\theta$. Moreover, if $\theta\in[0,\pi]$ and $T\geq5/7$, by the inequality $\log x\leq x-1$, one has $L_j(\theta)\leq\frac{L^\star_j(\theta)}{T+2}$, where
\begin{align*}
L^\star_j(\theta)=2r\sin\theta-4+\frac{7}{19}((j+c+r\cos\theta)^2+(r\sin\theta- 2)^2).
\end{align*}

In light of the choice of $F_{c,r}(\theta)$ (for Dirichlet $L$-functions) in \cite[Definition 5.10]{BMOR20}, we shall use the following $F_{c,r}(\theta)$ for $\zeta_K(s)$.

\begin{defn}
\label{Defn5.10}
For $\theta\in[-\pi,\pi]$, we let  $\sigma=c+r\cos\theta$, with $c-r>-\frac{1}{2}$, and $t=r\sin\theta$. For $\sigma\geq1+\eta$, we define
$$
F_{c,r}(\theta)=n_K\log\zeta(\sigma)+\frac{1}{2}L_{-1}(\theta)+\log(T+2). 
$$
For  $-\eta\leq\sigma\leq1+\eta$, we define
\begin{align*}
F_{c,r}(\theta)&= n_K\log\zeta(1+\eta)+\frac{n_K(1+\eta-\sigma)+2}{4}L_1(\theta)+\frac{n_K(1+\eta-\sigma)+2}{2}\log(T+2)\\
&+\frac{1+\eta-\sigma}{2}\Big(\log\frac{d_K}{{(2\pi)}^{n_K}}\Big)+\log3.
\end{align*}
For  $\sigma<-\eta$, we define
\begin{align*}
F_{c,r}(\theta)&= n_K\log\zeta(1-\sigma)+\frac{1}{2}L_{-1}(\theta)+\log(T+2)+\frac{1-2\sigma}{2}\log\Big(\frac{d_K(T+2)^{n_K}}{(2\pi)^{n_K}}\Big)\\
&+\frac{(1-2\sigma+2[\sigma])n_K}{4}L_{1-[\sigma]}(\theta)+\frac{n_K}{2}\sum\limits_{j=1}^{-[\sigma]}L_{j-1}(\theta).
\end{align*}
\end{defn}
 We note that $F_{c,r}(\theta)$ is an even function of $\theta$ satisfying $F_{c,r}(\theta)\geq\frac{1}{N}\log|f_{N}(c+re^{i\theta})|$. In order to bound $F_{c,r}(\theta)$, following \cite{BMOR20}, for $c\in\mathbb{R}$ and $r>0$, we define 
\[
\theta_{y}=
\begin{cases}
0 & \text{if $c+r\leq y$;} \\
\arccos\frac{y-c}{r} & \text{if $c-r\leq y\leq c+r$;} \\
\pi & \text{if $y \leq c-r$.}
\end{cases}
\]
For  the sake of convenience, we define
\begin{align*}
\kappa_1=\int_{\theta_{1 + \eta}}^{\theta_{-\eta}}  \frac{1+ \eta - \sigma}{2} d \theta 
 +  \int_{\theta_{-\eta}}^\pi \frac{1 - 2\sigma}{2} d \theta ,
\end{align*}
For  $J_1,J_2\in\mathbb{N}$, we shall set
\begin{align*}
\kappa_2(J_1)=\frac{\pi}{4J_1}\Big(\log\zeta(c+r)+2\sum_{j=1}^{J_1-1} \log\zeta \Big(c+r\cos\frac{\pi j}{2J_1}\Big)\Big),
\end{align*}
and
\begin{align*}
\kappa_3(J_2)=\frac{\pi-\theta_{1-c}}{2J_2}\Big(\log\zeta(1-c+r)+2\sum_{j=1}^{J_2-1} \log\zeta\Big(1-c-r\cos\Big(\frac{\pi j}{J_2}+\Big(1-\frac{j}{J_2}\Big)\theta_{1-c}\Big)\Big)\Big).
\end{align*}
In addition, we define
\begin{align*}
&\kappa_4=\frac{1}{4}\int_{\theta_{1+\eta}}^{\theta_{-\eta}}(1+\eta-\sigma)L^\star_1(\theta)d\theta,\\
&\kappa_5=\frac{1}{4}\int_{\theta_{-\eta}}^{\theta_{-1/2}}(1-2\sigma)L^\star_1(\theta)d\theta.
\end{align*}

Similar to \cite[Proposition 5.13]{BMOR20}, we have the following proposition regarding the upper bound of $\int_{0}^\pi F_{c,r} (\theta) d\theta$.

\begin{proposition}\label{bd-Fcr}
Let $c,r$, and $\eta$ be positive real numbers satisfying 
\begin{equation}\label{long-cond-1}
-\frac{1}{2}< c-r <-\eta <1+\eta < c
\end{equation}
and $0<\eta\leq\frac{1}{2}.$ Then for $T \geq \frac{5}{7}$, we have 
\begin{align*}
\int_{0}^\pi F_{c,r} (\theta) d\theta
&  \leq 
    n_K \int_0^{\theta_{1 + \eta}} \log \zeta (\sigma) d  \theta 
  + \frac{1}{2(T+2)}  \int_0^{\theta_{1 + \eta}} L_{-1}^{\star} (\theta) d\theta 
  + {\theta_{1 + \eta}} \log (T+2)\\
&  +  n_K ( \log \zeta (1 + \eta)) (\theta_{-\eta} -\theta_{1 + \eta} )  + \Big( \log \frac{d_K(T+2)^{n_K}}{(2\pi)^{n_K}}  \Big) \kappa_1 \\
& +  \frac{n_K}{T+2}\kappa_4 + \frac{1}{2(T+2)} \int_{\theta_{1 + \eta}}^{\theta_{-\eta}} L_1^\star (\theta) d \theta 
  +({\theta_{-\eta}}   - {\theta_{1 + \eta}})\log (3(T+2))\\
& +  n_K \int_{\theta_{-\eta}}^\pi \log \zeta (1 - \sigma) d\theta  
 +\frac{1}{2(T+2)} \int_{\theta_{-\eta}}^\pi L_{-1}^\star (\theta ) d \theta
 + (\pi   -{\theta_{-\eta}})\log (T+2)\\
&  +   \frac{n_K}{T+2} \kappa_5.
\end{align*}
\end{proposition}

\begin{proof}
We first write
$$
\int_{0}^\pi F_{c,r} (\theta) d \theta = \int_{0}^{\theta_{1 + \eta}} F_{c,r} (\theta) d \theta+\int_{\theta_{1 + \eta}}^{\theta_{-\eta}}  F_{c,r} (\theta) d \theta +\int_{\theta_{-\eta}}^\pi F_{c,r} (\theta) d \theta.
$$
By the definition of $F_{c,r}(\theta)$, we have
\begin{align*}
\int_{0}^{\theta_{1 + \eta}} F_{c,r} (\theta) d \theta
&= n_K \int_0^{\theta_{1 + \eta}} \log \zeta (\sigma) d  \theta +  \frac{1}{2}  \int_0^{\theta_{1 + \eta}} L_{-1} (\theta) d\theta +  \int_0^{\theta_{1 + \eta}} \log (T+2) d\theta \\
&\leq  n_K \int_0^{\theta_{1 + \eta}} \log \zeta (\sigma) d  \theta 
+  \frac{1}{2(T+2)}  \int_0^{\theta_{1 + \eta}} L_{-1}^{\star} (\theta) d\theta 
+ \theta_{1 + \eta} \log (T+2).
\end{align*}

Secondly, we compute
\begin{align}\label{2nd-F}
 \begin{split}
\int_{\theta_{1 + \eta}}^{\theta_{-\eta}}  F_{c,r} (\theta) d \theta 
&=  n_K \int_{\theta_{1 + \eta}}^{\theta_{-\eta}} \log \zeta (1 + \eta) d \theta
+ \Big( \log \frac{d_K}{(2\pi)^{n_K}} \Big) \int_{\theta_{1 + \eta}}^{\theta_{-\eta}}  \frac{1+ \eta - \sigma}{2} d \theta + \log 3 \int_{\theta_{1 + \eta}}^{\theta_{-\eta}}   1 d \theta \\
& +  \int_{\theta_{1 + \eta}}^{\theta_{-\eta}} \frac{n_K(1+ \eta - \sigma) +2 }{4} L_1 (\theta) d \theta 
+ \log (T+2) \int_{\theta_{1 + \eta}}^{\theta_{-\eta}}\frac{n_K(1+ \eta - \sigma) +2 }{2} d \theta. 
  \end{split}
\end{align}
The first three integrals on the right of \eqref{2nd-F} are
\begin{align*}
n_K ( \log \zeta (1 + \eta)) (\theta_{-\eta} -\theta_{1 + \eta} )
 +\Big( \log \frac{d_K}{(2\pi)^{n_K}} \Big) \int_{\theta_{1 + \eta}}^{\theta_{-\eta}}  \frac{1+ \eta - \sigma}{2} d \theta 
+(\log 3)(\theta_{-\eta} -\theta_{1 + \eta} ).
\end{align*}
As $1 + \eta -\sigma \geq 0$ for $\theta \in [\theta_{1+\eta}, \theta_{-\eta}]$, it follows that the last two integrals on the right of \eqref{2nd-F} are
\begin{align*}
& \frac{n_K}{4}\int_{\theta_{1 + \eta}}^{\theta_{-\eta}}  (1+ \eta - \sigma)   L_1 (\theta) d \theta  
+ \frac{1}{2} \int_{\theta_{1 + \eta}}^{\theta_{-\eta}} L_1 (\theta) d \theta 
 + n_K \log (T+2) \int_{\theta_{1 + \eta}}^{\theta_{-\eta}} \frac{1+ \eta - \sigma}{2} d \theta \\
& + ({\theta_{-\eta}} - {\theta_{1 + \eta}})\log (T+2)\\
&\le  \frac{n_K}{(T+2)}\kappa_4
 + \frac{1}{2(T+2)} \int_{\theta_{1 + \eta}}^{\theta_{-\eta}} L_1^\star (\theta) d \theta  
 +  n_K \log (T+2) \int_{\theta_{1 + \eta}}^{\theta_{-\eta}} \frac{1+ \eta - \sigma}{2} d \theta  \\
& +({\theta_{-\eta}} - {\theta_{1 + \eta}}) \log (T+2).
\end{align*}

Lastly, we have
\begin{align}\label{3rd-F}
 \begin{split}
\int_{\theta_{-\eta}}^\pi F_{c,r} (\theta) d \theta
&=
n_K \int_{\theta_{-\eta}}^\pi \log \zeta (1 - \sigma) d\theta  
+ \frac{1}{2} \int_{\theta_{-\eta}}^\pi L_{-1} (\theta ) d \theta
 +\int_{\theta_{-\eta}}^\pi \log (T+2) d \theta \\
&   + \Big( \log \frac{d_K(T+2)^{n_K} }{(2\pi)^{n_K}}\Big)  \int_{\theta_{-\eta}}^\pi \frac{1 - 2\sigma}{2} d \theta   +  n_K \int_{\theta_{-\eta}}^{\theta_{-\frac{1}{2}}} \frac{1- 2\sigma}{4}L_1 (\theta) d \theta \\
&  + n_K  \sum_{j=1}^{\infty} \int_{\theta_{-j + \frac{1}{2}}}^{\theta_{-j - \frac{1}{2}}} \Big( \frac{1 - 2\sigma - 2j}{4} L_{j+1} (\theta) + \frac{1}{2} \sum_{k=1}^j L_{k-1} (\theta) \Big)d \theta.
  \end{split}
\end{align}
The first four integrals on the right of \eqref{3rd-F} are
\begin{align*}
& \leq   n_K \int_{\theta_{-\eta}}^\pi \log \zeta (1 - \sigma) d\theta  
 +\frac{1}{2(T+2)} \int_{\theta_{-\eta}}^\pi L_{-1}^\star (\theta ) d \theta
  + \log (T+2) (\pi   -{\theta_{-\eta}})\\
 &  +  \Big( \log \frac{d_K (T+2)^{n_K}}{(2\pi)^{n_K}}\Big)  \int_{\theta_{-\eta}}^\pi \frac{1 - 2\sigma}{2} d \theta .
\end{align*}
Note that as $-\frac{1}{2}< c-r$, we have $\theta_{-j+\frac{1}{2}} = \theta_{-j-\frac{1}{2}} = \pi$ for $j \geq 1$. Thus, the remaining integral and sum on the the right of \eqref{3rd-F} is
\begin{align*}
  n_K \int_{\theta_{-\eta}}^{\theta_{-\frac{1}{2}}} \frac{1- 2\sigma}{4}L_1 (\theta) d \theta \leq \frac{n_K}{T+2} \int_{\theta_{-\eta}}^{\theta_{-\frac{1}{2}}} \frac{1- 2\sigma}{4}L_1^\star (\theta) d \theta 
=   \frac{n_K}{T+2} \kappa_5 .
\end{align*}
Putting all the estimates together, we complete the proof.
\end{proof}

To control ``zeta integrals'' in the above proposition, we shall borrow two estimates from  \cite[Lemmata 5.14 and 5.15]{BMOR20} as follows.

\begin{lemma}\label{zeta-int}
Let $c,r$ and $\eta$ be positive real numbers, satisfying \eqref{long-cond-1}, and $J_1$ and $J_2$ be positive integers.  If $\theta_{1 + \eta} \leq 2.1$, then for $\sigma = c+ r\cos \theta$, one has
\begin{align*}
\int_0^{\theta_{1 + \eta}} \log \zeta (\sigma) d\theta \leq \frac{\log \zeta (1 + \eta) + \log \zeta (c)}{2} \Big(\theta_{1 + \eta} - \frac{\pi}{2} \Big) + \frac{\pi}{4 J_1} \log \zeta (c) + \kappa_2 (J_1).
\end{align*}
In addition, assuming further $r > 2c  -1$, one has
\begin{align*}
\int_{\theta_{-\eta}}^\pi \log \zeta(1 - \sigma) d \theta \leq \frac{\log \zeta (1 + \eta) + \log \zeta (c)}{2} (\theta_{1-c} - \theta_{-\eta}) + \frac{\pi - \theta_{1-c}}{2 J_2} \log \zeta (c) + \kappa_3 (J_2).
\end{align*}
\end{lemma}

\section{Completing the proof}\label{completingP}
Gathering \eqref{main-est} and Propositions \ref{Je-bd} and \ref{BT}, for
$$
-\frac{1}{2}<c-r<1-c < -\eta <0 <\frac{1}{4} \le \delta=2c- \sigma_1 -\frac{1}{2} < \frac{1}{2} < 1 < 1+\eta < c< \sigma_1 =c+ \frac{(c-1/2)^2}{r}< c+r,
$$
satisfying $\theta_{1 + \eta} \leq 2.1$,   we have
\begin{align} \label{semi-final}
 \begin{split}
&\Big| N_K (T)  - \frac{T}{\pi} \log \Big( d_K \Big( \frac{T}{2\pi e}\Big)^{n_K}\Big)  + \frac{r_1}{4}\Big|\\
& \leq  \frac{5}{2}  + |g_K(T)| + \frac{2n_K}{\pi} \log \zeta (\sigma_1) + \frac{\log \Big( \frac{1}{\sqrt{(c-1)^2 + T^2}}\frac{\zeta_K(c)}{\zeta_K(2c)} \Big)}{\log \frac{r}{c - \frac{1}{2}}} + \frac{1}{\pi \log \frac{r}{c- \frac{1}{2}}} \int_0^\pi F_{c,r}(\theta) d\theta + \frac{E_K(T,\delta)}{\pi},
  \end{split}
\end{align}
where  $ g_K(T)$ and $E_K(T,\delta)$ are defined as in \eqref{equ:g-K} and \eqref{def-EK-2}, respectively, and
\[
\log \frac{\zeta_K(c)}{\zeta_K(2c)} = \int_c^{2c} - \frac{\zeta'_K}{\zeta_K}  (\sigma) d\sigma \leq n_K \int_c^{2c}- \frac{\zeta'}{\zeta}  (\sigma) d\sigma  \leq n_K \log \frac{\zeta(c)}{\zeta(2c)}.
\]
Finally, using  \eqref{gK-bd}, Lemma \ref{EK-bd-final}, Proposition \ref{bd-Fcr}, and Lemma \ref{zeta-int} to bound \eqref{semi-final}  and recalling that $r_1+2r_2 =n_K$, for any $T_0\ge \frac{5}{7}$, we obtain
\begin{equation}\label{final-est}
\Big| N_K (T)  - \frac{T}{\pi} \log \Big( d_K \Big( \frac{T}{2\pi e}\Big)^{n_K}\Big)  + \frac{r_1}{4}\Big|\le   C_1  \log\Big(  \frac{d_K(T+2)^{n_K}}{(2\pi)^{n_K}}  \Big)  +C_2  n_K  +C_3
\end{equation}
whenever $T\ge T_0$, where
\begin{align*}
C_1= \kappa_1 \Big(\pi \log \frac{r}{c- \frac{1}{2}}\Big)^{-1}, 
\end{align*}
\begin{align*}
C_2&=\frac{1}{25 T_0} +  \frac{2}{\pi} \log \zeta (\sigma_1)+\frac{640\delta -112}{1536(3T_0-1)}
+\max\Big\{0,   \frac{856\delta -151}{1536(3T_0+ 2)} -  \frac{640\delta -112}{1536(3T_0-1)}  \Big\}+ \frac{1}{2^{10}} \\
&+ \Big(\pi \log \frac{r}{c- \frac{1}{2}}\Big)^{-1} 
\Big( \frac{\log \zeta (1 + \eta) + \log \zeta (c)}{2} \Big(\theta_{1 + \eta} - \frac{\pi}{2} \Big) + \frac{\pi}{4 J_1} \log \zeta (c) + \kappa_2 (J_1) \Big) \\
&+\Big(\pi \log \frac{r}{c- \frac{1}{2}}\Big)^{-1} 
\Big( \frac{\log \zeta (1 + \eta) + \log \zeta (c)}{2} (\theta_{1-c} - \theta_{-\eta}) + \frac{\pi - \theta_{1-c}}{2 J_2} \log \zeta (c) + \kappa_3 (J_2) \Big) \\
&+\Big(\pi \log \frac{r}{c- \frac{1}{2}}\Big)^{-1} 
\Big(( \log \zeta (1 + \eta)) (\theta_{-\eta} -\theta_{1 + \eta} )
+   \max\Big\{0,\frac{\kappa_4+\kappa_5}{T_0 +2} \Big\}  + \pi \log \frac{\zeta(c)}{\zeta(2c)} \Big),
\end{align*}
\begin{align*}
C_3&= \frac{5}{2} +\Big(\pi  \log \frac{r}{c- \frac{1}{2}}\Big)^{-1}\Big( \pi \log \Big(1+\frac{  2  }{T_0}\Big) + ({\theta_{-\eta}}   - {\theta_{1 + \eta}})\log 3  \Big) \\
&+    \max \Big\{0, \Big(\pi \log \frac{r}{c- \frac{1}{2}}\Big)^{-1} \Big(\frac{1}{2(T_0+2)} \Big( \int_0^{\theta_{1 + \eta}} L_{-1}^{\star} (\theta) d\theta 
+  \int_{\theta_{1 + \eta}}^{\theta_{-\eta}} L_1^\star (\theta) d \theta 
+ \int_{\theta_{-\eta}}^\pi L_{-1}^\star (\theta ) d \theta \Big) \Big)\Big\}.
\end{align*}

For $T_0=1$ and $T_0=10$, choosing $J_1=64$ and $J_2=39$, via a Maple numerical computation, we have the following table of admissible $(C_1,C_2,C_3)$.
\begin{table}[htbp] 
\centering
\begin{tabular}{ |c|c|c||c|c|c|c|c|   } 
 \hline
 \multicolumn{3}{|c||}{  } &  & \multicolumn{2}{c|}{ $T\ge 1$} & \multicolumn{2}{c|}{ $T\ge 10$} \\
 \hline  
 $c$ & $r$ & $\eta$ &  $C_1$ & $C_2$ & $C_3$  & $C_2$ & $C_3$ \\ 
 \hline
 1.000011314 & 1.064340602 & $4.2826451\cdot 10^{-6}$   & 0.22737  &  23.02528  &   4.51954   & 22.97204 & 3.30668  \\ 
 \hline  
 1.042877508 & 1.259860485 & 0.01737451737   & 0.24493   & 6.66558 &  4.21201 & 6.60397  & 3.12362 \\ 
 \hline 
 1.079779637 & 1.410370323 & 0.03441682600   & 0.26304   & 5.22032   & 4.08149  & 5.15251 & 3.05074  \\ 
  \hline 
 1.114294066 & 1.538391756 & 0.05247813411   & 0.28032   & 4.43521   &  4.00936  & 4.36214 & 
3.01124 \\ 
 \hline 
 1.145720440 & 1.645584376 & 0.07107039918   & 0.29590  &  3.93889  &  3.96852 & 3.86136 & 2.98903 \\ 
 \hline 
\end{tabular}
   \caption{Choices of parameters $(c,r,\eta)$ and resulting admissible $(C_1,C_2,C_3)$}\label{table2} 
\end{table}

One may find functioning Maple code at \url{https://arxiv.org/abs/2102.04663}

\section*{Acknowledgments}
The authors would like to thank Nathan Ng for the encouragement and discussion for this project. They are also thankful to the referees for making helpful comments and suggestions.

\begin{rezabib}

\bib{Ba18}{article}{
    AUTHOR = {Backlund, R. J.},
     TITLE = {\"{U}ber die {N}ullstellen der {R}iemannschen {Z}etafunktion},
   JOURNAL = {Acta Math.},
  FJOURNAL = {Acta Mathematica},
    VOLUME = {41},
      YEAR = {1916},
    NUMBER = {1},
     PAGES = {345--375},
      ISSN = {0001-5962},
   MRCLASS = {DML},
  MRNUMBER = {1555156},
       DOI = {10.1007/BF02422950},
       URL = {https://doi.org/10.1007/BF02422950},}

\bib{BMOR20}{article}{
AUTHOR={Bennett, M. A.},
AUTHOR={Martin, G.},
AUTHOR={O'Bryant, K.},
AUTHOR={Rechnitzer, A.},
TITLE = {Counting zeros of {D}irichlet {$L$}-functions},
   JOURNAL = {Math. Comp.},
  FJOURNAL = {Mathematics of Computation},
    VOLUME = {90},
      YEAR = {2021},
    NUMBER = {329},
     PAGES = {1455--1482},
      ISSN = {0025-5718},
   MRCLASS = {11N13 (11M20 11M26 11N37 11Y35 11Y40)},
  MRNUMBER = {4232231},
       DOI = {10.1090/mcom/3599},
       URL = {https://doi.org/10.1090/mcom/3599},
}

\bib{KN12}{article}{
    AUTHOR = {Kadiri, H.},
    AUTHOR = {Ng, N.},
     TITLE = {Explicit zero density theorems for {D}edekind zeta functions},
   JOURNAL = {J. Number Theory},
  FJOURNAL = {Journal of Number Theory},
    VOLUME = {132},
      YEAR = {2012},
    NUMBER = {4},
     PAGES = {748--775},
      ISSN = {0022-314X},
   MRCLASS = {11R42 (11M41)},
  MRNUMBER = {2887617},
MRREVIEWER = {St\'{e}phane R. Louboutin},
       DOI = {10.1016/j.jnt.2011.09.002},
       URL = {https://doi.org/10.1016/j.jnt.2011.09.002},
}

\bib{LMO}{article}{
    AUTHOR = {Lagarias, J. C.},
    AUTHOR = {Montgomery, H. L.},
    AUTHOR = {Odlyzko, A. M.},
     TITLE = {A bound for the least prime ideal in the {C}hebotarev density
              theorem},
   JOURNAL = {Invent. Math.},
  FJOURNAL = {Inventiones Mathematicae},
    VOLUME = {54},
      YEAR = {1979},
    NUMBER = {3},
     PAGES = {271--296},
      ISSN = {0020-9910},
   MRCLASS = {12A75 (10H08)},
  MRNUMBER = {553223},
MRREVIEWER = {Larry J. Goldstein},
       DOI = {10.1007/BF01390234},
       URL = {https://doi.org/10.1007/BF01390234},
}

\bib{LO}{article}{
    AUTHOR = {Lagarias, J. C.},
    AUTHOR = {Odlyzko, A. M.},
     TITLE = {Effective versions of the {C}hebotarev density theorem},
 BOOKTITLE = {Algebraic number fields: {$L$}-functions and {G}alois
              properties ({P}roc. {S}ympos., {U}niv. {D}urham, {D}urham,
              1975)},
     PAGES = {409--464},
      YEAR = {1977},
   MRCLASS = {12A75},
  MRNUMBER = {0447191},
MRREVIEWER = {Matti Jutila},
}

\bib{McC84}{article}{
    AUTHOR = {McCurley, K. S.},
     TITLE = {Explicit estimates for the error term in the prime number
              theorem for arithmetic progressions},
   JOURNAL = {Math. Comp.},
  FJOURNAL = {Mathematics of Computation},
    VOLUME = {42},
      YEAR = {1984},
    NUMBER = {165},
     PAGES = {265--285},
      ISSN = {0025-5718},
   MRCLASS = {11N13 (11-04 11Y35)},
  MRNUMBER = {726004},
MRREVIEWER = {Matti Jutila},
       DOI = {10.2307/2007579},
       URL = {https://doi.org/10.2307/2007579},
}

\bib{Ra59}{article}{
    AUTHOR = {Rademacher, H.},
     TITLE = {On the {P}hragm\'{e}n-{L}indel\"{o}f theorem and some applications},
   JOURNAL = {Math. Z.},
  FJOURNAL = {Mathematische Zeitschrift},
    VOLUME = {72},
      YEAR = {1959/1960},
     PAGES = {192--204},
      ISSN = {0025-5874},
   MRCLASS = {10.00 (30.00)},
  MRNUMBER = {0117200},
MRREVIEWER = {N. G. de Bruijn},
       DOI = {10.1007/BF01162949},
       URL = {https://doi.org/10.1007/BF01162949},
}

\bib{Ro41}{article}{
    AUTHOR = {Rosser, B.},
     TITLE = {Explicit bounds for some functions of prime numbers},
   JOURNAL = {Amer. J. Math.},
  FJOURNAL = {American Journal of Mathematics},
    VOLUME = {63},
      YEAR = {1941},
     PAGES = {211--232},
      ISSN = {0002-9327},
   MRCLASS = {10.0X},
  MRNUMBER = {3018},
MRREVIEWER = {R. D. James},
       DOI = {10.2307/2371291},
       URL = {https://doi.org/10.2307/2371291},
}

\bib{Tr15}{article}{
    AUTHOR = {Trudgian, T. S.},
     TITLE = {An improved upper bound for the error in the zero-counting
              formulae for {D}irichlet {$L$}-functions and {D}edekind
              zeta-functions},
   JOURNAL = {Math. Comp.},
  FJOURNAL = {Mathematics of Computation},
    VOLUME = {84},
      YEAR = {2015},
    NUMBER = {293},
     PAGES = {1439--1450},
      ISSN = {0025-5718},
   MRCLASS = {11M06 (11M26 11R42)},
  MRNUMBER = {3315515},
MRREVIEWER = {Caroline L. Turnage-Butterbaugh},
       DOI = {10.1090/S0025-5718-2014-02898-6},
       URL = {https://doi.org/10.1090/S0025-5718-2014-02898-6},
}

\end{rezabib}

%
%
%
%
%
%
%
%
%
%
%
%
%
%

\end{document}